\documentclass{amsart}
\usepackage{uniform}

\title[New Families Satisfying the Uniform Boundedness Principle]
      {New families satisfying the Dynamical Uniform
        Boundedness Principle over function fields}

\author{John R. Doyle}
\address{Department of Mathematics, Oklahoma State University,
Stillwater, OK} 
\email{john.r.doyle@okstate.edu}

\author{Xander Faber}
\address{Institute for Defense Analyses, Center for Computing Sciences, Bowie, MD}
\email{awfaber@super.org}

\begin{document}

%%%%%%%%%%%%%%%%%%%%%%%%%%%%%%%%%%%%%%%%%%%%%%%%%%%%%%%%%%%%%%%%%%%%%%%%%%%%%%%%
%%%%%%%%%%%%%%%%%%%%%%%%%%%%%%%%%%%%%%%%%%%%%%%%%%%%%%%%%%%%%%%%%%%%%%%%%%%%%%%%

\begin{abstract}
  We extend a technique, originally due to the first author and
  Poonen, for proving cases of the Strong Uniform Boundedness
  Principle (SUBP) in algebraic dynamics over function fields of
  positive characteristic. The original method applied to unicritical
  polynomials for which the characteristic does not divide the
  degree. We show that many new 1-parameter families of polynomials
  satisfy the SUBP, including the family of all quadratic polynomials
  in even characteristic. We also give a new family of non-polynomial,
  non-Latt\`es rational functions that satisfies the SUBP.
\end{abstract}
\maketitle

%%%%%%%%%%%%%%%%%%%%%%%%%%%%%%%%%%%%%%%%%%%%%%%%%%%%%%%%%%%%%%%%%%%%%%%%%%%%%%%%

\section{Introduction}

Our goal  is to extend a technique of the first author and Poonen
for proving cases of the Strong Uniform Boundedness Conjecture in arithmetic
dynamics over function fields. For notation, let $K$ be a field and $f \in K(z)$
a nonconstant rational function. Define the set of $K$-rational
\textbf{preperiodic points} of $f$ to be
\[
   \PrePer(f,K) := \{x \in \PP^1(K) \ : \ \text{$x$ has finite forward orbit
     under $f$}\}.
\]

\begin{UBP}
  Let $K$ be any field with algebraic closure $\bar K$. Let $d >1$, and let $\cF
  \subset \bar K(z)$ be a set of rational functions of degree $d$.
  \begin{itemize}
    \item We say that $\cF$ satisfies the \textbf{Uniform Boundedness Principle}
      (\textbf{UBP}) over $K$ if there is a bound $A = A(\cF,K)$ such that
      ${\# \PrePer(f, K) \le A}$ for each $f \in \cF(K)$.
    \item We say that $\cF$ satisfies the \textbf{Strong Uniform Boundedness
      Principle} (\textbf{SUBP}) over $K$ if for every $D \ge 1$ there is a
      bound $B = B(\cF,D)$ such that ${\# \PrePer(f,L) \le B}$ for every extension
      $L/K$ of degree $D$ and every $f \in \cF(L)$.
  \end{itemize}
\end{UBP}

In order for either of these Uniform Boundedness Principles to hold over a
field~$K$, the number of $K$-rational preperiodic points for maps in the family
$\cF$ must be finite. Thus, for example, (S)UBP will not be satisfied over an
algebraically closed field. On the other hand, when $K$ is a number field,
Northcott showed that any rational function of degree at least $2$ has only
finitely many $K$-rational preperiodic points \cite{Northcott_1950}. In this
setting, Morton and Silverman have conjectured that the set $\cF = \Rat_d(\bar
\QQ)$ of degree-$d$ rational functions with algebraic number coefficients
satisfies the Strong Uniform Boundedness Principle over $\QQ$
\cite[p. 100]{Morton_Silverman_1994}. We give a summary of the current state of
knowledge in the next section.

Let $k$ be a field. Throughout this article, a \textbf{function field over $k$}
is any finite extension of the rational function field $k(\T)$. Equivalently, a
function field over $k$ is the field of rational functions of some integral
$k$-curve. We refer to $k$ as the \textbf{constant subfield} of the function
field.

The primary test case for many ideas in dynamics on the projective line is the
family of quadratic polynomials $f_c(z) = z^2 + c$. The first author and Poonen
proved the Strong Uniform Boundedness Principle for this family over a function
field\footnote{For the SUBP to hold in this setting, one must generally exclude
those parameters $c$ lying in the constant subfield.} in characteristic
different from $2$ \cite{Doyle_Poonen_gonality}. Though global in spirit, the
argument for the positive characteristic case in \cite{Doyle_Poonen_gonality}
leaned heavily on an analysis of the dynamics of the map $f(z) = z^2 + \t$ on
the Julia set over the Laurent series field $\FF_q\Ls{1/\t}$. By generalizing
this dynamical setup, we can extend this argument to other 1-parameter families.

More precisely, take $f_\t \in \FF_q(\t)(z)$ with $d := \deg_z(f_\t) > 1$. We
can associate to $f_\t$ a countable collection of smooth algebraic
$\FF_q$-curves, known as {\it dynatomic curves}, whose $L$-rational points (for
an extension $L/\FF_q$) parameterize maps $f_\s$ with $\s \in L$ together with
a marked $L$-rational preperiodic point of $f_\s$. Each dynatomic curve $Z$ is
equipped with a canonical morphism $\varphi_Z \colon Z \to \PP^1$, projecting
onto the $t$-coordinate and forgetting the marked preperiodic point. See
Section~\ref{sec:dynatomic} for more details.  For convenience, we will say that
a property of dynatomic curves holds ``as $Z \to \infty$'' if that property
holds as $i \to \infty$ for any ordering $Z_1, Z_2, Z_3, \ldots$ of the set of
dynatomic curves. Consider the following statements:

\begin{enumerate}
  \item The degree of definition and the ramification index of the geometric
    points in the fiber $\varphi_Z^{-1}(\infty)$ remain uniformly bounded as $Z
    \to \infty$.
  
  \item The degree of $\varphi_Z \colon Z \to \PP^1$ tends to infinity
    as $Z \to \infty$.
      
  \item There is $\rd \ge 1$ such that $\#Z(\FF_{q^\rd}) \to \infty$ as
    $Z \to \infty$.
      
  \item The gonality\footnote{The \textbf{gonality} of an irreducible $k$-curve
  $C$ is the minimum degree of a nonconstant $k$-morphism $C \to \PP^1$.} of $Z$ tends to
    infinity as $Z \to \infty$.
      
  \item For any extension $k/\FF_q$, and for any function field $K$ over $k$,
    the Strong Uniform Boundedness Principle over $K$ holds for the set
   \[
     \cF = \{ f_\s \ : \ \s \in \overline{K} \text{ and } f_\s \text{ is
       not $\overline{K}$-conjugate to an element of $\overline{k}(z)$}\}.
   \]
\end{enumerate}

For brevity, we will say ``the Strong Uniform Boundedness Principle over $K$
holds for the family $f_\t$'' if (5) is true.  The arguments in
\cite{Doyle_Poonen_gonality} give the chain of implications
  \[
    (1)+(2) \ \Longrightarrow \ (3) \ \Longrightarrow \ (4) \Longrightarrow \ (5).
  \]
We provide sufficient conditions for (1) and (2) to hold in
Section~\ref{sec:family}. Applying them to some special cases, we are able to
execute the above plan for new families of polynomials, including the first case
where the characteristic of the ground field divides the degree.

\begin{theorem}
  \label{thm:ubp_polynomials}
  Let $\FF_q$ be a finite field, let $d > 1$, and let $\alpha_1, \ldots,
  \alpha_d \in \FF_q[\t]$ be distinct polynomials. Set
  \[
      f_\t(z) = (z - \alpha_1) \cdots (z - \alpha_d). 
      \]
  Assume that for each $i \ne j$, the following inequality is true:
  \[
     \deg(\alpha_i - \alpha_j) + \sum_{\ell \ne j} \deg(\alpha_\ell -
     \alpha_j) > \max_\ell \deg(\alpha_\ell).
  \]
  Then, for any extension $k/\FF_q$ and any function field $K$ over $k$, the
  Strong Uniform Boundedness Principle over $K$ holds for the family $f_\t$.
\end{theorem}

For example, we could enumerate the elements of $\FF_q$ as ${\varepsilon_1,
  \ldots, \varepsilon_q}$ and set $\alpha_i = \varepsilon_i \t$. The hypothesis
of the theorem is satisfied, and we have ${f_\t(z) = z^q - \t^{q-1}z}$. Cases
where the characteristic of the ground field divides the degree of the family
have been a sticking point in much previous work.

The first author and Poonen proved that the Strong Uniform Boundedness Principle
holds for the set of all quadratic polynomials over a function field of
characteristic $p \ne 2$ \cite{Doyle_Poonen_gonality}. That approach focused on
the family $z^2 + c$, which does not capture a general quadratic in even
characteristic. We now prove the Strong Uniform Boundedness Principle for this
family in all positive characteristics by studying a family of quadratic
polynomials which is more appropriate in characteristic $2$.
  
\begin{corollary}
  \label{cor:quadratic}
  Let $k$ be any field, and let $K$ be a function field over $k$. For each $D
  \ge 1$, there exists a bound $B = B(D)$ such that $\#\PrePer(f,L) \le B$ for
  any finite extension $L/K$ of degree at most $D$ and any quadratic polynomial
  $f$ with coefficients in $L$ that is not conjugate to a polynomial with
  coefficients in $\overline{k}$.
\end{corollary}

\begin{proof}
  We may assume that the characteristic of $k$ is positive, as the
  characteristic-zero case was handled in \cite{Doyle_Poonen_gonality}.
  
  By Theorem~\ref{thm:ubp_polynomials}, the Strong Uniform Boundedness Principle
  holds for the family ${f_\t(z) = z(z+\t)}$. Fix $D \ge 1$, and set $D' =
  2D$. Then there exists a bound $B = B(D')$ such that for any function field
  $L' / K$ with $[L' \colon K ] \le D'$ and any $\s \in L' \smallsetminus
  \overline{k}$, we have
  \[
      \# \PrePer(f_\s, L') \le B.
  \]
  We claim that $B$ is the bound we seek.

  Let $L/K$ be an extension of degree $D$, and let ${g(z) = az^2 + bz +
    c}$ be a quadratic polynomial over $L$. Replacing $g(z)$ with $a g(z/a)$
  allows us to assume $a = 1$ without affecting the number of $L$-rational
  preperiodic points. Set $L' = L(u)$, where $u$ is a solution to the equation
  $z^2 + (b-1)z + c = 0$, and define
  \[
    h(z) = g(z+u) - u = z(z+2u+b). 
  \]
  Since $[L' \colon K] \le 2[L \colon K] \le D'$, and since $h$
  is a member of the quadratic family $f_\t$, we find that
  \[
  \#\PrePer(g,L) \le \#\PrePer(h,L') \le B. \qedhere
  \]
\end{proof}

A rational function is of \textbf{polynomial} type if it has a totally ramified
fixed point.  Aside from trivial examples like finite sets, we are aware of
three sets of \textit{non-polynomial} type rational functions for which the
Strong Uniform Boundedness Principle has been shown to hold over number fields
or function fields: Latt\`es maps, functions with a bounded amount of bad
reduction, and twists of a single rational function with nontrivial automorphism
group. We give further details and provide references on these examples in
Section~\ref{sec:background}. We close this introduction with an example of an
algebraic family of non-polynomial type rational functions that avoids all of
these special cases.

\begin{theorem}
  \label{thm:ubp_rational}
  Let $k$ be a field of characteristic $p > 0$, and let $K$ be a function field
  over $k$. Let $d \ge 2$ and $e \le d -2$ be nonnegative integers such that $p$
  does not divide $d$. Define
  \[
      f_\t(z) = \frac{z^d - \t}{z^e}.
  \]
  Then the Strong Uniform Boundedness Principle over $K$ holds for $f_\t$.
\end{theorem}

\begin{remark}
  \label{rem:doyle_poonen}
By taking $e = 0$ and replacing $\t$ with $-\t$, Theorem~\ref{thm:ubp_rational}
applies to the family $z^d + \t$ with $p \nmid d$, thus allowing us to recover
the main results of \cite{Doyle_Poonen_gonality} in positive characteristic.
\end{remark}

\subsection*{Acknowledgments}
The first author was partially supported by NSF grant DMS-2112697. We thank
Laura DeMarco and Nicole Looper for their feedback on this work, and we thank
the anonymous referee for helpful comments.

%%%%%%%%%%%%%%%%%%%%%%%%%%%%%%%%%%%%%%%%%%%%%%%%%%%%%%%%%%%%%%%%%%%%%%%%%%%%%%%%
%%%%%%%%%%%%%%%%%%%%%%%%%%%%%%%%%%%%%%%%%%%%%%%%%%%%%%%%%%%%%%%%%%%%%%%%%%%%%%%%

\section{The dynamical uniform boundedness conjecture}
\label{sec:background}

Inspired by the strong uniform boundedness conjecture for torsion points on
elliptic curves (later proved by Merel \cite{merel:1996}), Morton and Silverman
posed the following conjecture, which we state in terms of the Strong Uniform
Boundedness Principle.  For a field $K$, write $\Rat_d(K)$ for the set of all
degree-$d$ rational functions defined over $K$.

\begin{conjecture}
[{Dynamical Uniform Boundedness Conjecture; \cite[p. 100]{Morton_Silverman_1994}}]
\label{conj:ubc}
For each integer $d \ge 2$, the family $\Rat_d(\QQbar)$ satisfies the Strong
Uniform Boundedness Principle over $\bbQ$.
\end{conjecture}

Conjecture~\ref{conj:ubc} is currently wide open. To illustrate the difficulty
in proving this conjecture, we note that just the $d = 4$ case of
Conjecture~\ref{conj:ubc} is sufficient to prove Merel's theorem, since a point
$P$ on an elliptic curve in Weierstrass form is a torsion point if and only if
its $x$-coordinate is preperiodic under the ``duplication map" $x(P) \mapsto
x(2P)$, which is a degree-$4$ rational function on $\PP^1$.

There has been considerable progress, however, especially when restricting to
polynomial maps. We summarize the current state of affairs by describing subsets
of $\Rat_d(\QQbar)$ for which the (S)UBP has been proven.
\begin{enumerate}
\item Looper has shown that, assuming a strong version of the $abc$-conjecture
  for number fields, the set $\Poly_d(\bar\QQ)$ of degree-$d$ polynomials
  satisfies the UBP over every number field \cite{looper, looper:2021}.
\item A Latt\`es map is an endomorphism of $\PP^1$ of degree at least~2 that is
  covered by an endomorphism of an elliptic curve \cite[\S6.4]{silverman:2007}.
  Using deep results of Mazur, Kamienny, and Merel, one can show that the SUBP
  holds for the set of Latt\`es maps over a number field
  \cite[\S6.7]{silverman:2007}.
\item
  For a number field $K$ and a rational function $f \in K(z)$, we say that $f$
  has {\it good reduction} at the prime ideal $\frakp$ of the ring of integers
  $\calO_K$ if the reduction\footnote{To define the reduction modulo $\frakp$, one should
      normalize $f$ so that the minimum $\frakp$-adic valuation of the
      coefficients is $0$.} of $f$ modulo $\frakp$ has degree equal to the
  degree of $f$. Then for any integer $s \ge 0$, it is known that
  \[
  \hspace*{1.2cm}
   \Rat_{d,s}(\bar\QQ) :=
   \bigcup_{K/\bbQ\text{ finite}}
   \left\{f \in \Rat_d(K)\ :\
   \begin{tabular}{l}
     \text{$f$ has good reduction away} \\
     \text{from a set of $s$ primes of $\calO_K$}
   \end{tabular}
   \right\}
\]
satisfies the SUBP over $\QQ$. There is a substantial literature on this topic;
we mention the articles \cite{benedetto:2007, canci/paladino:2016,
  Morton_Silverman_1994, narkiewicz/pezda:1997}, but we recommend \cite[Remark
  3.16]{silverman:2007} for a more comprehensive list.

\item It was conjectured in \cite{flynn/poonen/schaefer:1997} that a quadratic
  polynomial in $\QQ[z]$ cannot have rational periodic points of period larger
  than $3$, and there is a significant amount of evidence to support this; see
  \cite{flynn/poonen/schaefer:1997, hutz/ingram:2013, morton:1998,
    stoll:2008}. Poonen proved that the set
\[
   \cF :=
   \left\{f \in \QQ[z]\ :\
   \begin{tabular}{l}
     \text{$\deg(f) = 2$ and $f$ has no rational} \\
     \text{point of period greater than $3$}
   \end{tabular}
   \right\}
\]
satisfies the UBP over $\QQ$ \cite{poonen:1998}. In fact, the explicit bound
$\#\PrePer(f,\QQ) \le 9$ is given for $f \in \cF$.

\item Recall that the \textbf{automorphism group} of a rational function $f \in
  K(z)$ is the subgroup of elements $g \in \PGL_2(\bar K)$ such that $g^{-1}
  \circ f \circ g = f$. Manes has proved that the set
\[
   \cF :=
   \left\{f \in \QQ(z)\ :\
   \begin{tabular}{l}
     \text{$\deg(f) = 2, \ \Aut(f) \ne 1$, and $f$ has no} \\
     \text{rational point of period greater than $4$}
   \end{tabular}
   \right\}
\]
satisfies the UBP over $\QQ$ \cite{Manes_Quadratic_2008}. Moreover, Manes showed
that the bound ${\#\PrePer(f,\QQ) \le 12}$ holds for all $f \in \cF$.

\item For a fixed number field $K$ and rational function $f \in K(z)$ of degree
  at least $2$, let $[f]$ denote the set of all rational functions in $K(z)$
  that are $\PGL_2(\Kbar)$-conjugate to $f$. Levy, Manes, and Thompson have
  shown that $[f]$ satisfies the UBP over $K$
  \cite{levy/manes/thompson:2014}.\footnote{The content of this statement comes
    from the fact that a rational function $f$ may admit infinitely many nontrivial
    \textbf{twists}: rational functions conjugate to $f$ over $\Kbar$, but not
    over $K$.}
\end{enumerate}

We also recommend \cite{doyle:2018quad, doyle:2020,
  doyle/faber/krumm:2014, ingram:2019} for additional results toward uniform
boundedness over number fields.

A version of Conjecture~\ref{conj:ubc} for function fields is also believed to
be true, once one removes certain obvious sources of counterexamples. Recall that elements of
the group $\PGL_2$ act on rational functions by conjugation.

\begin{conjecture}
[Dynamical Uniform Boundedness Conjecture for function fields]
\label{conj:ubcff}
Let $k$ be a field, and let $K$ be the function field of an integral curve over
$k$. For each integer $d \ge 2$, the family
\[
\Rat_d(\bar K) \smallsetminus \PGL_2(\bar K).\Rat_d(\bar k)
\]
satisfies the Strong Uniform Boundedness Principle over $K$.
\end{conjecture}

We must remove the rational functions that are conjugate to an element of $\bar
k(z)$ --- the so-called \textbf{isotrivial} functions --- because they can have
infinitely many preperiodic points if the field $k$ is infinite.

Let $k$ be a field, and let $K$ be a function field over $k$.
\begin{enumerate}
\item Looper showed that if $k$ has characteristic zero and a strong version of
  the $abc$-conjecture holds for $K$, then the family
  \[
  \Poly_d(\Kbar) \smallsetminus \mathrm{Aff}_2(\bar K).\Poly_d(\bar k)
  \]
  satisfies the UBP over $K$ \cite{looper, looper:2021}. Here
  $\mathrm{Aff}_2(\bar K)$ is the subgroup of $\PGL_2(\bar K)$ that preserves
  the point at infinity.

\item The SUBP holds for the set of Latt\`es maps over $K$ that do not arise
  from an elliptic curve whose $j$-invariant is algebraic over $k$
  \cite{nguyen-saito,Poonen_gonality}.
  
 \item Benedetto showed that the space $\Poly_{d,s}(\Kbar)$ of degree-$d$
   polynomials with good reduction away from a set of $s$ places satisfies the
   SUBP over $K$ \cite{benedetto:2005}. Assuming $k$ has characteristic zero,
   Canci has shown that $\Rat_{d,s}\big(k(\T)\big)$ satisfies the UBP over
   $k(\T)$ \cite{canci:2015}.
   
  \item The first author and Poonen showed that if the characteristic of $k$
    does not divide $d$, then the family
\[
    \calF(d) := \{z^d + c : c \in \Kbar \smallsetminus \kbar\}
\]
satisfies the SUBP over $K$ \cite{Doyle_Poonen_gonality}. (Compare
Remark~\ref{rem:doyle_poonen}.)
\end{enumerate}

%%%%%%%%%%%%%%%%%%%%%%%%%%%%%%%%%%%%%%%%%%%%%%%%%%%%%%%%%%%%%%%%%%%%%%%%%%%%%%%%
%%%%%%%%%%%%%%%%%%%%%%%%%%%%%%%%%%%%%%%%%%%%%%%%%%%%%%%%%%%%%%%%%%%%%%%%%%%%%%%%
  
\section{Sufficient local conditions for the SUBP}
\label{sec:sufficient}

We begin with a technical lemma that is restrictive in its hypotheses, but still
general enough to aid with all of our examples. Then we use local considerations
to derive consequences for the dynatomic curves associated to a one-parameter
family.

%%%%%%%%%%%%%%%%%%%%%%%%%%%%%%%%%%%%%%%%%%%%%%%%%%%%%%%%%%%%%%%%%%%%%%%%%%%%%%%%

\subsection{Nonarchimedean preliminaries}

\begin{lemma}
  \label{lem:finite}
  Let $k$ be a nonarchimedean field with nontrivial absolute value $|\cdot|$.
  Suppose that $D \subset k$ is an open disk and $f \colon D \to D$ is an analytic map
  such that $f(D) \subsetneq D$, and such that $D$ contains a fixed point of
  $f$. Then $f$ has finitely many preperiodic points, and each preperiodic point
  eventually maps to the fixed point.
\end{lemma}

\begin{proof}
  Without loss of generality, we may assume that $k$ is algebraically
  closed. After a possible change of coordinate, we may assume that $D =
  D(0,1)^-$ is the open unit disk and $f(0) = 0$. Write $f(z) = \sum_{n \ge 1} a_n
  z^n$ with $a_n \in k$. Since $f(D) \subsetneq D$, there is $r < 1$ such that
  $|a_n| \le r$ for all $n \ge 0$. For $x \in D$, we have 
  \[
     |f(x)| = |x| \cdot \left| \sum_{n \ge 1} a_n x^{n-1} \right| \le |x|
     \max_{n \ge 1} |a_nx^{n-1}| \le r|x|.
  \]
  By induction, we find $|f^j(x)| \le r^j|x|$ for all $j \ge 1$, and hence,
  every point in $D$ converges to $0$ under iteration. In particular, every
  preperiodic point must eventually map to $0$.

  By Weierstrass Preparation, an analytic function on an affinoid domain has
  only finitely many zeros. Thus, there is $\varepsilon > 0$ such that the only
  solution to $f(z) = 0$ in the disk $D(0,\varepsilon)^-$ is the origin. Fix a
  positive integer~$j$ such that $r^j < \varepsilon$. Then ${|f^j(x)| \le r^j|x|
    < \varepsilon}$. So if $x$ is preperiodic for $f$, then we must have $f^j(x)
  = 0$. Since $j$ depends only on $f$, we conclude there are finitely many
  preperiodic points for $f$, and they all eventually map to the fixed point in
  $D$.
\end{proof}

\begin{lemma}
  \label{lem:symbolic}
  Let $k$ be a complete nonarchimedean field with nontrivial absolute value
  $|\cdot|$.  Let $f \in k(z)$ be a separable rational function of degree $d >
  1$ with an attracting fixed point at $\infty$. Write $D_\infty$ for the
  maximal open disk in the immediate basin of $\infty$, and write $Y$ for the
  complement of $D_\infty$. Assume that $f^{-1}(Y)$ is a disjoint union of $d$
  closed disks. Then the following conclusions hold:
  \begin{itemize}
    \item Every preperiodic point either eventually maps to $\infty$ or else
      lies in the Julia set for $f$.
    \item There exists a finite extension $k' / k$ such that every preperiodic
      point for $f$ lies in $\PP^1(k')$.
    \item All of the finite preperiodic points for $f$ lie inside a disk about
      the origin.
  \end{itemize}
\end{lemma}

\begin{proof}
  First, we argue that $D_\infty$ is well defined. Take $D_\infty$ to be the
  union of all open disks $D$ about $\infty$ such that $f^n(D)$ converges to
  $\infty$ as $n$ grows. There exists at least one such disk since $\infty$ is
  an attracting fixed point, so $D_\infty$ is nonempty. Since $f$ has $d+1 > 1$
  fixed points counted with multiplicity, and since an attracting fixed point
  must have multiplicity $1$, there is some fixed point of $f$ not in
  $D_\infty$. It follows that $D_\infty \ne \PP^1_k$, and hence $D_\infty$ is an
  open disk.
  
  Now we show that $f(D_\infty) \subset D_\infty$. Since there is a fixed point
  of $f$ that does not lie in $D_\infty$, we see that $f(D_\infty) \ne \PP^1_k$.
  Thus $f(D_\infty)$ is an open disk that contains $\infty$, and that is
  contained in the basin of $\infty$. By maximality, it follows that
  ${f(D_\infty) \subset D_\infty}$.

  Set $Y = \PP^1_k \smallsetminus D_\infty$. Then $Y$ is a closed disk. Write
  $f^{-1}(Y) = X_1 \cup \cdots \cup X_d$, where each $X_i$ is a closed disk. We
  claim that each $X_i$ is properly contained in $Y$. Take $x \in \bigcup
  X_i$. If $x \not\in Y$, then $x \in D_\infty$. By the preceding paragraph,
  $f(D_\infty) \subset D_\infty$, so we find that $f(x) \not\in Y$, a
  contradiction. So each $X_i \subset Y$. If $X_i = Y$ for some $i$, then $X_j
  \subset X_i$ for $j \ne i$, which contradicts the hypothesis that the $X_j$
  are pairwise disjoint. Thus, $X_i \subsetneq Y$.

  Next we argue that $f(D_\infty) \ne D_\infty$. Since the $X_i$ are disjoint,
  we may enlarge each disk slightly to obtain disks $X_i' \subset Y$ that are
  pairwise disjoint and contain no pole of $f$, and such that $Y \subsetneq
  f(X_i')$. Choose any $x \in D_\infty$ that lies in the intersection of the
  $f(X_i')$. Then $f^{-1}(x)$ consists of $d$ distinct points inside $Y$. In
  particular, $x \in D_\infty \smallsetminus f(D_\infty)$.

  Write $Y = D(b,r)$, the disk of radius $r$ about some $b \in \bar k$. We now
  argue that $r \in |\bar k^\times|$. Let us change coordinates in order to move
  $D_\infty$ to $D_0 := D(0,1/r)^-$. More precisely, set
  \[
      g(z) = \frac{1}{f(b+1/z) - b}.
  \]
  Then the previous paragraph shows $g$ maps $D_0$ strictly into itself. Now $g$
  is given by a series $\sum g_n z^n$ on $D_0$. If the radius of convergence of
  this series were larger than $1/r$, then by continuity, there would be a
  slightly larger disk $D_0' \supset D_0$ such that $g(D_0') \subset D_0$. This
  would violate maximality of $D_\infty$. Thus, the radius of convergence of the
  series is $1/r$. But $g$ is a rational function, so the obstruction to
  extending the domain of convergence of the series is a pole of $g$. As every
  pole lies in $\bar k^\times$, we see that $r \in |\bar k^\times|$, as desired.

  We may now define the extension $k' / k$. Take $Y = D(b,r)$ as before. Let
  $a_1, \ldots, a_d \in \bar k$ be the distinct solutions to $f(z) = b$. Then
  $X_i = D(a_i,r_i)$. Choose $c_i \in \bar k$ such that $|c_i| = r_i$; this is
  possible since the radius of $Y$ lies in $|\bar k^\times|$. Define
  \[
  \calP = \{x \in \PP^1(\bar k) \smallsetminus \{\infty\} \colon f(x) \in D_\infty
  \text{ and $x$ is preperiodic for $f$}\}.
  \]
  Since $f(D_\infty) \subsetneq D_\infty$, Lemma~\ref{lem:finite} shows that the
  set $\calP$ is finite. (This proves the final conclusion of the lemma.) Define
  $k'$ to be the extension of $k$ given by adjoining $\calP \cup \{a_1, \ldots,
  a_d, c_1, \ldots, c_d\}$.

  In general, if $g \colon D \to D'$ is a separable injective $k'$-analytic map
  from a disk $D$ onto a disk $D'$, then the inverse of $g$ exists as an
  analytic function and is defined over $k'$. Applying this observation to each
  of the inverses ${h_i \colon Y \to D(a_i, r_i)}$, we see that the solutions to
  $f^n(z) = \infty$ are $k'$-rational for every $n \ge 0$. Additionally,
  this setup allows us to use the argument in the proof of Proposition~4.1 of
  \cite{Kiwi_Quadratic_Puiseux_2014} to conclude that the Julia set of $f$ is
  contained inside $\PP^1(k')$, and that the Fatou set of $f$ is the immediate
  basin of attraction of $\infty$. In particular, every preperiodic point for
  $f$ either lies in the Julia set or else eventually maps to $\infty$, and all
  such preperiodic points lie in $\PP^1(k')$. 
\end{proof}

\begin{remark}
  The strategy of Proposition~4.1 in \cite{Kiwi_Quadratic_Puiseux_2014} shows
  that the dynamics of $f$ on the Julia set $\cJ(f)$ is conjugate to the
  left-shift map on the space of sequences of $d$ symbols. The map is given by
  sending $x \in \cJ(f)$ to its itinerary $(i_0, i_1, i_2, \ldots)$, where
  $f^j(x) \in X_{i_j}$ for each~${j \ge 0}$.
\end{remark}

\begin{remark}
  We can rephrase the hypothesis of the lemma in terms of ramification of the
  morphism of associated analytic spaces. If $k$ is a complete nonarchimedean
  field, we write $\Berk^1_k$ for the associated analytic space in the sense of
  Berkovich. For a rational function $f \in k(z)$, we write $\cR_f$ for the
  ramification locus of $f$ inside $\Berk^1_k$. If $f$ is separable, then $\cR_f
  \ne \Berk^1_k$ and $f$ is locally injective on $\Berk^1_k \smallsetminus
  \cR_f$. By \cite[Thm.~6.3.2]{Berkovich_Etale_1993}, $f^{-1}(Y)$ is a disjoint
  union of $d$ disks if and only if $f(\cR_f) \subset \mathbf{D}_\infty$. Here
  $\mathbf{D}_\infty$ is the closure of $D_\infty$ inside $\Berk^1_k$. See
  \cite{Faber_Berk_RamI_2013} for additional mapping properties related to the
  ramification locus.
\end{remark}

%%%%%%%%%%%%%%%%%%%%%%%%%%%%%%%%%%%%%%%%%%%%%%%%%%%%%%%%%%%%%%%%%%%%%%%%%%%%%%%%

\subsection{Dynatomic curves}
\label{sec:dynatomic}
Let $k$ be any field. Given a family of rational functions $f_\t(z) \in
k(\t)(z)$ with $\deg_z(f_\t) = d \ge 2$, one can define a dynamical analogue of
the classical modular curves, typically referred to as \textbf{dynatomic
  curves}. First, write $f_\t(z) = a(z) / b(z)$ with $a, b \in
k[\t][z]$. Without loss of generality, we may assume that $\gcd_z(a,b) = 1$ and
that the coefficients of $a$ and $b$ have no common factor in $k[\t]$ of
positive degree. This determines $a, b$ up to a common scalar in $k^\times$. Set
$A(X,Y) = Y^da(X/Y)$ and $B(X,Y) = Y^d b(X/Y)$, and let $F = (A,B)$ be the
induced family of morphisms on $\PP^1_k$. Inductively define polynomials $A_m$
and $B_m$ by
\[
  A_0 = X, \quad B_0 = Y, \quad A_m = A_{m-1}(A,B), \quad B_m = B_{m-1}(A,B).
  \]
We set $F^m = (A_m, B_m)$; this corresponds to the $m$-th iterate of the
morphism~$F$.
Let $H$ be an irreducible factor of
\[
A_{m+n}B_m - A_{m}B_{m+n} \in k[\t][X,Y]
\]
for some $m \ge 0$ and $n \ge 1$. Then $H$ is a polynomial in $X,Y,\t$,
homogeneous in $X$ and $Y$, and so its vanishing defines an algebraic curve
inside $\PP^1_k \times \Aff^1_k$. In this paper, a \textbf{dynatomic curve} $Z$
is the normalization of the projective closure of the curve $\{H = 0\}$. Two
dynatomic curves are considered distinct if they have different associated
polynomials.  The rational function $\t$ gives a morphism $\varphi_Z \colon Z \to
\PP^1$; by the \textbf{degree of $Z$} we will mean the degree of the morphism
$\varphi_Z$. For a field extension $K/k$ and $\s \in \PP^1(K)$, elements of the
fiber $\varphi_Z^{-1}(\s)$ correspond to points $x \in \PP^1(\bar K)$ which
satisfy $f_\s(x)^{m+n} = f_\s^m(x)$. (They may also satisfy this equation for
smaller values of $m$ and $n$.)

For each place $v$ of $k(\t)$, we can form the associated completion $k(\t)_v$
with respect to $v$. These are fields of formal Laurent series. For example, if
$v = \ord_\infty$, then $k(\t)_v = k\Ls{1/\t}$, while if $v = \ord_a$ for some
$a \in k$, then $k(\t)_v = k\Ls{\t-a}$. Given a family $f_\t \in k(\t)(z)$ of
rational functions parameterized by $\t$, we abuse notation and write $f = f_\t$
for the associated rational function defined over the function field $k(\t)$ or
any of its completions $k(\t)_v$.

\begin{proposition}
  \label{prop:local}
  Let $f_\t \in \FF_q(\t)(z)$ be a separable family of rational functions
  satisfying ${d := \deg_z(f_\t) > 1}$. Suppose the following hypotheses hold:
  \begin{itemize}
    \item At the place $v = \ord_\infty$, the point at $\infty$ is an attracting
      fixed point for $f$, and the pre-image of the complement of the maximal
      open disk in the immediate basin of $\infty$ is a union of $d$ disjoint
      closed disks.
    \item At each place $v \ne \ord_\infty$, the finite preperiodic points of
      $f$ are integral.
  \end{itemize}
  Then $\infty$ is a superattracting fixed point for $f$, and the following
   also hold:
  \begin{enumerate}
  \item\label{item:rs} There exist integers $\rd, \ri \ge 1$ such that for any
    dynatomic curve $Z$, the points of the fiber of ${\varphi_Z \colon Z \to
      \PP^1}$ over $\infty$ all lie in $Z(\FF_{q^\rd})$ and have ramification
    index at most $\ri$.
  \item The degrees of the dynatomic curves for $f_\t$ tend to infinity as $Z
    \to \infty$.
  \item The gonalities of the dynatomic curves for $f_\t$ tend to infinity as $Z
    \to \infty$.
  \item For each dynatomic curve $Z$, the irreducible components of the base
    extension $Z_{\FF_{q^\rd}}$ are geometrically irreducible, where $\rd$ is the
    integer in (\ref{item:rs}).
  \end{enumerate}
\end{proposition}

\begin{remark}
  The integers $\rd, \ri$ in the statement of the proposition are the residue
  degree and ramification index of the extension $k' / k$ from
  Lemma~\ref{lem:symbolic}. The proof of the lemma is constructive, so one can
  calculate $\rd$ and $\ri$ explicitly.
\end{remark}

\begin{proof}[Proof of Proposition~\ref{prop:local}]
  We begin by arguing that $\infty$ is superattracting for $f$. Suppose
  otherwise, and let $\lambda \in \FF_q(\t)$ be the multiplier at
  $\infty$. Since $\infty$ is an attracting fixed point at the place
  $\ord_\infty$, we have $\ord_\infty(\lambda) > 0$. By the product formula,
  there exists a place $v$ such that $v(\lambda) < 0$; that is, $\infty$ is
  $v$-adically repelling. Repelling fixed points belong to the Julia set
  $\cJ_v(f)$, and the backward orbit of any Julia point is dense in the Julia
  set. It follows that there is some $\alpha \in \overline{\FF_q(\t)}$ such that
  $f^n(\alpha) = \infty$ for some $n > 0$ and $v(\alpha) < 0$. This contradicts
  our assumption that all preperiodic points except $\infty$ are $v$-adically
  integral.

  Let $Z$ be a dynatomic curve with associated polynomial
  \[
  H(X,Y) = H_0 X^\ell + H_1 X^{\ell-1}Y + \cdots + H_{\ell-1} XY^{\ell-1} + H_\ell Y^\ell,
  \]
  where each $H_i \in \FF_q[\t]$. The degree of the morphism $\varphi_Z \colon Z
  \to \PP^1$ is $\ell$.  If $H_0 = 0$, then $H$ is divisible by $Y$. As $H$ is
  irreducible, this means that, after possibly dividing by an element of
  $\FF_q^\times$, we have $H = Y$. The first three conclusions of the
  proposition are indifferent to the behavior of any one curve, and the final
  conclusion is clearly true for $\{Y = 0\}$. In what remains, we may suppose
  that $H_0 \ne 0$.  Consider the univariate polynomial
  \[
  h(z) = z^\ell + \frac{H_1}{H_0}z^{\ell-1} + \cdots + \frac{H_{\ell-1}}{H_0}z +
  \frac{H_\ell}{H_0} \in \FF_q(\t)[z].
  \]
  The zeros of $h$ are preperiodic points for the map $f$ in
  $\Aff^1(\overline{\FF_q(\t)})$.

  Let $v = \ord_\infty$, and view $h$ as living in $\FF_q(\t)_v[z]$. By
  Lemma~\ref{lem:symbolic}, the roots of $h$ are defined over some finite
  extension $K / \FF_q(\t)_v$. Let $\rd$ and $\ri$ be the degree of the residue
  extension and ramification index of $K / \FF_q(\t)_v$, respectively. Observe
  that $\rd,\ri$ depend only on $f$ and not on $Z$. The points of
  $\varphi_Z^{-1}(\infty)$ are defined over $\FF_{q^\rd}$ and have ramification
  index at most $\ri$. This completes the proof of conclusion (1).

  Now we claim that the degrees of the $H_i$ are uniformly bounded in terms of
  $\ell$ and the family $f_\t$. As the $H_i$ are polynomials over a finite
  field, this means there are only finitely many possibilities for the $H_i$,
  and hence finitely many dynatomic curves of degree $\ell$ over the $\t$-line.
  The hypothesis that the finite preperiodic points of $f$ are integral at all
  $v \ne \ord_\infty$ means $H_0 \in \FF_q^\times$. Without loss of generality,
  we may assume $H_0 = 1$.  We may apply the final conclusion of
  Lemma~\ref{lem:symbolic} to obtain a nonnegative integer $N$, depending only
  on $f$, such that $\ord_\infty(x) \ge -N$ for all roots $x$ of $h$.  Since the
  coefficients $H_i$ are symmetric polynomials in the roots of $h$, we find that
  \[
      \ord_\infty(H_i) \ge -iN \text{ for all } 1 \le i \le \ell.
  \]
  But $\ord_\infty = -\deg$, so we learn that $\deg(H_i) \le iN \le \ell N$. That
  is, the degrees of the coefficients $H_i$ are uniformly bounded. Conclusion
  (2) is now proved.

  Let $Z$ be a dynatomic curve and consider the morphism $\varphi_Z \colon Z \to
  \PP^1$. We have already shown that every point in the fiber over infinity is
  defined over $\FF_{q^\rd}$ and has ramification index at most~$\ri$. In particular,
  $\# Z(\FF_{q^\rd}) \ge \deg(\varphi_Z) / \ri$. If $\gamma$ is the gonality of
  $Z$, then we also know that $\# Z(\FF_{q^\rd}) \le \gamma(q^\rd+1)$ since every
  point of $\PP^1(\FF_{q^\rd})$ has at most $\gamma$ geometric points above it in
  $Z$. Combining these inequalities shows that
  \[
  \gamma \ge \frac{\deg(\varphi_Z)}{\ri(q^\rd+1)}.
  \]
  Since the degree of $\varphi_Z$ tends to infinity as $Z \to \infty$, so does the
  gonality of $Z$, as desired in conclusion (3).

  Finally, let $Z' = Z_{\FF_{q^\rd}}$, the base extension of $Z$ to
  $\FF_{q^\rd}$. Let $W$ be an irreducible component of $Z'$.  If $W$ is not
  geometrically irreducible, then the base extension $W_{\bar \FF_q}$ has
  irreducible components $V \ne V'$ that are $\Gal(\bar \FF_{q^\rd} /
  \FF_{q^\rd})$-conjugate. Note that $W$ projects onto $\PP^1$ via the map
  $\varphi_{Z'}$. In particular, $W$ has an $\FF_{q^\rd}$-rational point since
  every point in the fiber of $Z'$ over infinity is $\FF_{q^\rd}$-rational. But
  then $V \cap V'$ contains each of these rational points, all of which must be
  singular on $W$.  This is impossible since $Z'$ is smooth. It follows that $W$
  is geometrically irreducible, and conclusion (4) holds.
\end{proof}

The connection between gonality and the Strong Uniform Boundedness Principle is
given by the following proposition. The proof is a straightforward
generalization of the argument for Theorem~1.7 in \cite{Doyle_Poonen_gonality}.
  
\begin{proposition}
  \label{prop:SUBP}
Let $f_\t \in \FF_q(\t)(z)$ be a family of rational functions satisfying
$\deg_z(f_\t) > 1$, and let $k$ be any field containing $\FF_q$. If the
gonalities of the dynatomic curves for $f_\t$ tend to infinity in any ordering,
then the family $f_\t$ satisfies the Strong Uniform Boundedness Principle over
any function field over $k$.
\end{proposition}

\begin{remark}
The proof of Proposition~\ref{prop:SUBP} can be used to give an upper bound for
the constant $B = B(\cF,D)$ in the statement of the SUBP, which depends on $q$
and $D$ as well as the quantities $\rd$, $\ri$, and $N$ appearing in the proof
of Proposition~\ref{prop:local}. The bound obtained in this way, which is larger
than $Dq^{(Dq^\rd)^2}$, is rather cumbersome and unlikely to be anywhere near
optimal.
\end{remark}

%%%%%%%%%%%%%%%%%%%%%%%%%%%%%%%%%%%%%%%%%%%%%%%%%%%%%%%%%%%%%%%%%%%%%%%%%%%%%%%%
%%%%%%%%%%%%%%%%%%%%%%%%%%%%%%%%%%%%%%%%%%%%%%%%%%%%%%%%%%%%%%%%%%%%%%%%%%%%%%%%

\section{New Families}
\label{sec:family}

As promised in the introduction, we now give sufficient conditions to be able to
utilize Proposition~\ref{prop:SUBP}. 

\begin{proposition}
  \label{prop:technical}
  Let $\FF_q$ be a finite field, and let $f_\t(z)$ be a family of rational
  functions of the form
  \[
     f_\t(z) = \frac{a(z)}{b(z)},
  \]
  where
  \begin{enumerate}
  \item $a$ and $b$ are coprime monic polynomials with coefficients in
    $\FF_q[\t]$;
  \item $\deg(a) > \deg(b) + 1$; and
  \item  $a$ is separable.
  \end{enumerate}
  For $v = \ord_\infty$, let $R$ be the maximum $v$-adic absolute value of the
  roots of $a$. Then we further assume that
  \begin{enumerate}
    \setcounter{enumi}{3}
  \item $R > 1$;
  \item the roots of $b$ all lie in the $v$-adic disk $D(0,R)$;
  \item there exists a root of $a(z) - zb(z)$ with $v$-adic absolute value $R$;
  \item for each root $\alpha$ of $a$, there is a disk $D(\alpha,r_\alpha)$ that
    maps onto $D(0,R)$; and
  \item the disks $D(\alpha,r_\alpha)$ and $D(\beta,r_\beta)$ are disjoint if
    $\alpha,\beta$ are distinct roots of $a$.
  \end{enumerate}
  Then for any field $k$ containing $\FF_q$ and any function field $K$ over $k$,
  the family $f_\t$ satisfies the Strong Uniform Boundedness Principle over $K$.
\end{proposition}

\begin{proof}
By Proposition~\ref{prop:SUBP}, it suffices to show that the gonalities of the
dynatomic curves for $f_\t$ tend to infinity. We accomplish this by showing that
the hypotheses of Proposition~\ref{prop:local} are satisfied.

Let $v \ne \ord_\infty$ be a place of $\FF_q(\t)$. Let $x \in
\overline{\FF_q(\t)}$ be an element with $|x|_v > 1$. Then $x$ is larger than
any root of $a$ or $b$, as both polynomials are monic with $v$-adic integral
coefficients. Thus,
\[
  |f(x)|_v = |x|_v^{\deg(a) - \deg(b)} \ge |x|_v^2 > |x|_v.
\]
It follows that $x$ is not preperiodic. That is, all finite preperiodic points
of $f$ are $v$-adically integral.

Next, take $v = \ord_\infty$. Write $A = \deg(a)$ and $B = \deg(b)$. Since $a,b$
are coprime and $A - B \ge 2$, the point at infinity is a superattracting fixed
point. Let $R > 1$ be the maximum absolute value of a root of $a$ in $\overline{
  \FF_q(\t)}$. Set $U = \PP^1 \smallsetminus D(0,R)$. We claim that
$U = D_\infty$, the maximal open disk in the immediate basin of $\infty$. If
$|x|_v > R$, then the fact that all roots of $a$ and $b$ lie in $D(0,R)$ shows
that
\[
   |f(x)|_v = \frac{|a(x)|_v}{|b(x)|_v} = |x|_v^{A-B} \ge |x|_v^2 > R|x|_v.
\]
So $U \subset D_\infty$. By hypothesis, the polynomial $a(z) - zb(z)$ has a
root with absolute value $R$, which means $f$ has a fixed point of absolute
value $R$. That is, no disk larger than $U$ lies in $D_\infty$. Thus, $U =
D_\infty$, as desired.

Let $Y = D(0,R)$. Hypotheses (3), (7), and (8) of the lemma say that $f^{-1}(Y)$
consists of $\deg(f) = \deg(a)$ pairwise disjoint disks. This completes the
proof.
\end{proof}

\begin{remark}
  The applicability of Proposition~\ref{prop:technical} may depend on the choice of
  coordinate. For example, the proposition does not apply to
  \[
     f_\t(z) = (z - \t^2)(z - \t^2 - \t),
   \]
  since condition (8) fails, though outside of characteristic $5$ it does apply to the conjugate
  \[
     g_\t(z) := f_\t(z+\t^2) - \t^2 = z^2 - \t z - \t^2.
     \]
\end{remark}

We now show that the families from the introduction satisfy the hypotheses of
Proposition~\ref{prop:technical}. Let $\FF_q$ be a fixed finite field
throughout.

\begin{example}
Fix an integer $d > 1$, and let $\alpha_1, \ldots, \alpha_d \in \FF_q[\t]$ be
distinct polynomials. Set
\[
  f_\t(z) = (z - \alpha_1) \cdots (z - \alpha_d). 
\]
Assume that for each $i \ne j$, the following inequality is true:
\begin{equation}
  \label{eq:weird_inequality}
   \deg(\alpha_i - \alpha_j) + \sum_{\ell \ne j} \deg(\alpha_\ell -
   \alpha_j) > \max_\ell \deg(\alpha_\ell).
\end{equation}
We now prove Theorem~\ref{thm:ubp_polynomials} by verifying the conditions of
Proposition~\ref{prop:technical}. Set $M = \max_\ell \deg(\alpha_\ell)$ for ease of
notation.  
\begin{enumerate}
\item[(1-3, 5)] Clear, since $f_\t$ is a polynomial in $z$ of degree $d > 1$,
  and the $\alpha_i$ are distinct.
\item[(4)] Let $v = \ord_\infty$. Note that $M > 0$, for otherwise every
  $\alpha_i$ is constant and \eqref{eq:weird_inequality} does not hold. Since
  $f$ is in factored form, it is immediate that the maximum absolute value of a
  root is $R = |\t|_v^M > 1$.
\item[(6)] All roots of $f(z)$ have nonpositive $v$-adic valuation. As $f$ is
  monic and at least one of the $\alpha_i$ is nonconstant, each segment of the
  Newton polygon lies below the $x$-axis. It follows that $f(z)-z$ and $f(z)$
  have the same Newton polygon. In particular, $f(z)-z$ has a root with $v$-adic
  absolute value~$R$.
\item[(7)] Fix $j$ and define
  \[
     r_j = \frac{R}{\prod_{\ell \ne j} |\alpha_\ell - \alpha_j|_v}.
  \]
  Equation~\eqref{eq:weird_inequality} is equivalent to the assertion that $r_j
  < |\alpha_i - \alpha_j|_v$ for all $i \ne j$.  Set $x = \alpha_j + y$, where
  $|y|_v \le r_j$. Then
  \begin{align*}
    |f(x)|_v &= |y|_v \cdot \prod_{\ell \ne j} |\alpha_j - \alpha_\ell + y|_v \\
    &= |y|_v \cdot \prod_{\ell \ne j} |\alpha_j - \alpha_\ell |_v, 
  \end{align*}
  since $|y|_v \le r_j < |\alpha_j  - \alpha_\ell |_v$ for all $\ell \ne j$. That is,
  \[
  |f(x)|_v = \frac{|y|_v}{r_j} R.
  \]
  As $|y|_v$ varies from $0$ to $r_j$, we obtain elements $f(\alpha_j+y)$ with
  absolute value from $0$ to $R$. It follows that $f$ maps $D(\alpha_j,r_j)$
  onto $D(0,R)$, as desired.
\item[(8)] If $i \ne j$, then the disks $D(\alpha_i, r_i)$ and $D(\alpha_j,r_j)$
  are disjoint. Indeed, by the ultrametric inequality it suffices to check that
  $|\alpha_i - \alpha_j| > \max(r_i,r_j)$, and this was already observed in our
  proof of condition~(7).
\end{enumerate}
\end{example}

\begin{example}
Fix integers $d \ge 2$ and $e \le d-2$ such that $p = \mathrm{char}(\FF_q)$
does not divide $d$. Define
\[
    f_\t(z) = \frac{z^d - \t}{z^e}.
\]
We once again verify the conditions of Proposition~\ref{prop:technical} are met,
thus proving Theorem~\ref{thm:ubp_rational}.
\begin{enumerate}
\item[(1-2)] Clear.
\item[(3)] Since $p$ does not divide $d$, the numerator $z^d - \t$ is separable.
\item[(4)] Let $v = \ord_\infty$. The roots of the numerator of $f$ are of the
  form $\varepsilon \t^{1/d}$ for some $d$-th root of unity $\varepsilon$. These
  all have $v$-adic absolute value $R = |\t|_v^{1/d} > 1$.
\item[(5)] Clear.
\item[(6)] The Newton polygon of $(z^d+\t) - z\cdot z^e$ has a single segment, so
  that all fixed points of $f$ have absolute value $R = |\t|_v^{1/d}$.
\item[(7)] Let $x = \alpha + y$ for some root $\alpha$
  of $z^d-\t$ and some $y$ with $|y|_v \le R^{2+e-d}$.  Then
\[
  f(x) =  x^{-e} \sum_{i = 1}^d \binom{d}{i} \alpha^{d-i} y^i.
\]
Since $p$ does not divide $d$, the $i=1$ term in the sum strictly dominates
the others, so we find
\[
|f(x)|_v = R^{d-e-1} |y|_v = \frac{|y|_v}{R^{2+e-d}}\ R.
\]
That is, as $|y|_v$ varies from $0$ to $R^{2+e-d}$, we obtain elements
$f(\alpha+y)$ with absolute value from $0$ to $R$. It follows that $f$ maps
$D(\alpha,R^{2+e-d})$ onto $D(0,R)$.
\item[(8)] The fact that $p$ does not divide $d$ implies that any pair of
  distinct roots of $z^d - \t$ are at distance $R$ from each other. Since $R > 1
  \ge R^{2+e-d}$, we find that the disks $D(\alpha,R^{2+e-d})$ are disjoint as
  $\alpha$ varies through the roots of $z^d - \t$.
\end{enumerate}
\end{example}

%%%%%%%%%%%%%%%%%%%%%%%%%%%%%%%%%%%%%%%%%%%%%%%%%%%%%%%%%%%%%%%%%%%%%%%%%%%%%%%%
%%%%%%%%%%%%%%%%%%%%%%%%%%%%%%%%%%%%%%%%%%%%%%%%%%%%%%%%%%%%%%%%%%%%%%%%%%%%%%%%

\bibliography{uniform}

\providecommand{\bysame}{\leavevmode\hbox to3em{\hrulefill}\thinspace}
\providecommand{\MR}{\relax\ifhmode\unskip\space\fi MR }
\providecommand{\MRhref}[2]{%
  \href{http://www.ams.org/mathscinet-getitem?mr=#1}{#2}
}
\providecommand{\href}[2]{#2}
\begin{thebibliography}{10}

\bibitem{benedetto:2005}
Robert~L. Benedetto, \emph{Heights and preperiodic points of polynomials over
  function fields}, Int. Math. Res. Not. (2005), no.~62, 3855--3866.
  \MR{2202175}

\bibitem{benedetto:2007}
\bysame, \emph{Preperiodic points of polynomials over global fields}, J. Reine
  Angew. Math. \textbf{608} (2007), 123--153. \MR{2339471}

\bibitem{Berkovich_Etale_1993}
Vladimir~G. Berkovich, \emph{\'{E}tale cohomology for non-{A}rchimedean
  analytic spaces}, Inst. Hautes \'Etudes Sci. Publ. Math. (1993), no.~78,
  5--161. \MR{MR1259429 (95c:14017)}

\bibitem{canci:2015}
Jung~Kyu Canci, \emph{Preperiodic points for rational functions defined over a
  rational function field of characteristic zero}, New York J. Math.
  \textbf{21} (2015), 1295--1310. \MR{3441644}

\bibitem{canci/paladino:2016}
Jung~Kyu Canci and Laura Paladino, \emph{Preperiodic points for rational
  functions defined over a global field in terms of good reduction}, Proc.
  Amer. Math. Soc. \textbf{144} (2016), no.~12, 5141--5158. \MR{3556260}

\bibitem{doyle:2018quad}
John~R. Doyle, \emph{Preperiodic points for quadratic polynomials with small
  cycles over quadratic fields}, Math. Z. \textbf{289} (2018), no.~1-2,
  729--786. \MR{3803810}

\bibitem{doyle:2020}
\bysame, \emph{Preperiodic points for quadratic polynomials over cyclotomic
  quadratic fields}, Acta Arith. \textbf{196} (2020), no.~3, 219--268.
  \MR{4161298}

\bibitem{doyle/faber/krumm:2014}
John~R. Doyle, Xander Faber, and David Krumm, \emph{Preperiodic points for
  quadratic polynomials over quadratic fields}, New York J. Math. \textbf{20}
  (2014), 507--605. \MR{3218788}

\bibitem{Doyle_Poonen_gonality}
John~R. Doyle and Bjorn Poonen, \emph{Gonality of dynatomic curves and strong
  uniform boundedness of preperiodic points}, Compos. Math. \textbf{156}
  (2020), no.~4, 733--743. \MR{4065068}

\bibitem{Faber_Berk_RamI_2013}
Xander Faber, \emph{Topology and geometry of the {B}erkovich ramification locus
  for rational functions, {I}}, Manuscripta Math. \textbf{142} (2013), no.~3-4,
  439--474. \MR{3117171}

\bibitem{flynn/poonen/schaefer:1997}
E.~V. Flynn, Bjorn Poonen, and Edward~F. Schaefer, \emph{Cycles of quadratic
  polynomials and rational points on a genus-{$2$} curve}, Duke Math. J.
  \textbf{90} (1997), no.~3, 435--463. \MR{1480542 (98j:11048)}

\bibitem{hutz/ingram:2013}
Benjamin Hutz and Patrick Ingram, \emph{On {P}oonen's conjecture concerning
  rational preperiodic points of quadratic maps}, Rocky Mountain J. Math.
  \textbf{43} (2013), no.~1, 193--204. \MR{3065461}

\bibitem{ingram:2019}
Patrick Ingram, \emph{Canonical heights and preperiodic points for certain
  weighted homogeneous families of polynomials}, Int. Math. Res. Not. IMRN
  (2019), no.~15, 4859--4879. \MR{3988672}

\bibitem{Kiwi_Quadratic_Puiseux_2014}
Jan Kiwi, \emph{Puiseux series dynamics of quadratic rational maps}, Israel J.
  Math. \textbf{201} (2014), no.~2, 631--700.

\bibitem{levy/manes/thompson:2014}
Alon Levy, Michelle Manes, and Bianca Thompson, \emph{Uniform bounds for
  preperiodic points in families of twists}, Proc. Amer. Math. Soc.
  \textbf{142} (2014), no.~9, 3075--3088. \MR{3223364}

\bibitem{looper}
Nicole Looper, \emph{The uniform boundedness and dynamical {L}ang conjectures
  for polynomials}, preprint (2021). \texttt{arXiv:2105.05240v2}.

\bibitem{looper:2021}
Nicole~R. Looper, \emph{Dynamical uniform boundedness and the
  {$abc$}-conjecture}, Invent. Math. \textbf{225} (2021), no.~1, 1--44.
  \MR{4270662}

\bibitem{Manes_Quadratic_2008}
Michelle Manes, \emph{{$\mathbb{Q}$}-rational cycles for degree-2 rational maps
  having an automorphism}, Proc. Lond. Math. Soc. (3) \textbf{96} (2008),
  no.~3, 669--696. \MR{MR2407816 (2009a:14029)}

\bibitem{merel:1996}
Lo{\"{\i}}c Merel, \emph{Bornes pour la torsion des courbes elliptiques sur les
  corps de nombres}, Invent. Math. \textbf{124} (1996), no.~1-3, 437--449.
  \MR{1369424 (96i:11057)}

\bibitem{morton:1998}
Patrick Morton, \emph{Arithmetic properties of periodic points of quadratic
  maps. {II}}, Acta Arith. \textbf{87} (1998), no.~2, 89--102. \MR{1665198
  (2000c:11103)}

\bibitem{Morton_Silverman_1994}
Patrick Morton and Joseph~H. Silverman, \emph{Rational periodic points of
  rational functions}, Internat. Math. Res. Notices (1994), no.~2, 97--110.
  \MR{1264933 (95b:11066)}

\bibitem{narkiewicz/pezda:1997}
W.~Narkiewicz and T.~Pezda, \emph{Finite polynomial orbits in finitely
  generated domains}, Monatsh. Math. \textbf{124} (1997), no.~4, 309--316.
  \MR{1480362}

\bibitem{nguyen-saito}
Khac~Viet Nguyen and Masa-Hiko Saito, \emph{$d$-gonality of modular curves and
  bounding torsions}, preprint (1996). \texttt{arXiv:alg-geom/9603024v1}.

\bibitem{Northcott_1950}
D.~G. Northcott, \emph{Periodic points on an algebraic variety}, Ann. of Math.
  (2) \textbf{51} (1950), 167--177. \MR{MR0034607 (11,615c)}

\bibitem{poonen:1998}
Bjorn Poonen, \emph{The classification of rational preperiodic points of
  quadratic polynomials over {${\bf Q}$}: a refined conjecture}, Math. Z.
  \textbf{228} (1998), no.~1, 11--29. \MR{1617987 (99j:11076)}

\bibitem{Poonen_gonality}
\bysame, \emph{Gonality of modular curves in characteristic {$p$}}, Math. Res.
  Lett. \textbf{14} (2007), no.~4, 691--701. \MR{2335995}

\bibitem{silverman:2007}
Joseph~H. Silverman, \emph{The arithmetic of dynamical systems}, Graduate Texts
  in Mathematics, vol. 241, Springer, New York, 2007. \MR{2316407
  (2008c:11002)}

\bibitem{stoll:2008}
Michael Stoll, \emph{Rational 6-cycles under iteration of quadratic
  polynomials}, LMS J. Comput. Math. \textbf{11} (2008), 367--380. \MR{2465796
  (2010b:11067)}

\end{thebibliography}
\bibliographystyle{amsplain}

\end{document}